\documentclass[12pt]{article}
\usepackage{latexsym}

\usepackage{amsthm,amssymb,amsmath,graphics,graphicx}
\usepackage{epstopdf}
\usepackage{tikz}
\usepackage[margin=1in]{geometry}
\usepackage{enumerate}  
\usepackage[english]{babel}

\newtheorem{theorem}{Theorem}[section]

\newtheorem{prop}[theorem]{Proposition}

\newtheorem{lemma}[theorem]{Lemma}
\newtheorem{cor}[theorem]{Corollary}

\newtheorem{conj}[theorem]{Conjecture}
\newtheorem{defi}[theorem]{Definition}
\newtheorem{fact}[theorem]{Fact}

\let\oldmarginpar\marginpar
\renewcommand\marginpar[1]{\-\oldmarginpar[\raggedleft\footnotesize #1]%
{\raggedright\footnotesize #1}}

\newcommand\Osq{\hspace{1mm}\square\hspace{1mm}}

\newcommand{\diam}{{\rm diam}}

\def\vertex(#1){\put(#1){\bowtiele*{2}}}
\def\vertexo(#1){\put(#1){\bowtiele{2}}}
\def\vert(#1){\put(#1){\bowtiele*{1.5}}}
\def\verto(#1){\put(#1){\bowtiele{1.5}}}
\def\lab(#1)#2{\put(#1){\makebox(0,0)[c]{#2}}}
\begin{document}


\title{Pebbling on Graph Products and other Binary Graph Constructions}

\author{
John Asplund\\
{\small Department of Technology and Mathematics,}\\
{Dalton State College,} \\
{\small Dalton, GA 30720, USA} \\
{\small jasplund@daltonstate.edu}\\
\\
Glenn Hurlbert
\thanks{Research partially supported by Simons Foundation Grant \#246436.}\\
{\small Department of Mathematics and Applied Mathematics,}\\
{Virginia Commonwealth University,}\\
{\small Richmond, VA 23284-2014, USA} \\
{\small ghurlbert@vcu.edu}\\
\\
Franklin Kenter\\
{\small Department of Mathematics,}\\
{United States Naval Academy,}\\
{\small 
Annapolis, MD 21402, USA} \\
{\small kenter@usna.edu}\\
 }

\date{}
\maketitle

~\\

\begin{abstract}
Pebbling on graphs is a two-player game which involves repeatedly moving a pebble from one vertex to another by removing another pebble from the first vertex. 
The pebbling number $\pi(G)$ is the least number of pebbles required so that, regardless of the initial configuration of pebbles, a pebble can reach any vertex.
Graham conjectured that the pebbling number for the cartesian product, $G \Osq H$, is bounded above by $\pi(G) \pi(H)$. 
We show that $\pi(G \Osq H) \le 2\pi(G) \pi(H)$ and, more sharply, that $\pi(G \Osq H) \le (\pi(G)+|G|) \pi(H)$. 
Furthermore, we provide similar results for other graph products and graph operations. 
\end{abstract}



\section{Introduction}

{\it Pebbling on graphs} is a two-player game on a connected graph, first mentioned by Lagarias and Saks and made popular in the literature by Chung \cite{chung1989pebbling}. 
The game starts when Player 1 chooses a {\it root} vertex $r$ and places pebbles on the vertices.
Then Player 2 may make a {\it pebbling move} by choosing two pebbles at the same vertex, moving one of them to an adjacent vertex, and removing the other. 
Player 2 wins if he or she can get a pebble to $r$ after any non-negative number of pebbling moves.
The {\it pebbling number} of the graph, denoted $\pi(G)$, is the fewest number of pebbles such that, regardless of the initial configuration and root, Player 2 can win. 
A good reference for pebbling results, variations, and applications can be found in \cite{hurlHand}.


This study focuses on providing upper bounds for the pebbling numbers of various graph products and other graph constructions. 
For two graphs, $G$ and $H$, define the ``box'' (or cartesian) product $G \Osq H$ to be the graph with $V(G \Osq H) = V(G) \times V(H)$ and $E(G \Osq H) = \{ \{ (g, h) , (g', h')\} \colon g=g' \text{ and } hh' \in E(H) \text{ or } h=h' \text{ and } gg' \in E(G) \}$.  
Graham conjectured that the box product of two connected graphs obeys the following:
\begin{conj}[Graham's Conjecture {\normalfont\cite{chung1989pebbling}}]
\label{conjBox}
 \[\pi(G\Osq H)\le \pi(G)\pi(H).\]
\end{conj}
The conjecture would be tight if true \cite{chung1989pebbling}. 
It is still open and has only been solved in a handful of cases, including products of paths \cite{chung1989pebbling}, products of cycles \cite{herscovici2003graham, snevily20002}, products of trees \cite{moews}, products of fan graphs and products of wheel graphs \cite{feng2002pebbling} and products between graphs from certain families and graphs with the so-called 2-pebbling property \cite{chung1989pebbling,snevily20002,wang2009graham}. 
It has also been verified for graphs of high minimum degree \cite{czyghurl2006girth}.


We make progress toward this conjecture from a different perspective. 
Instead of focusing on select families of graphs for which we can verify the conjecture, we aim to find the best constant $\beta$ for which we can prove that $\pi(G \Osq H) \le \beta \pi (G) \pi(H)$ is universally true. 
In fact, in Section \ref{box}, we prove that for any connected graphs $G$ and $H$, \[ \pi(G \Osq H) \le 2 \pi (G) \pi(H).\]

Later, we extend this technique to derive similar bounds for other graph products and graph constructions which we will carefully define in the next sections.
Our results include bounds for the pebbling number on the strong graph product (Section \ref{strong}), the cross (or ``categorical'') product (Section \ref{cross}), and coronas (Section \ref{corona}). 
Beyond the work on Graham's conjecture on the box product for graphs, there are few published results for graph pebbling on other products.  
A variant of Graham's conjecture using the strong graph product instead of the box product was mentioned as a ``problem'' in \cite{wang2001pebbling} but no further progress on this variation is known to the authors. 
Note that there are upper bounds provided by Kim and Kim \cite{KimKim2} for the lexicographic product of $G$ and $H$. 



\section{Preliminaries}

Throughout, we will assume that all graphs are simple. 
We use the following notation. 
For a graph $G = (V,E)$, $V = V(G)$ and $E = E(G)$ are the vertex and edge sets of $G$, respectively, and $|G|=|V(G)|$.
We write $g \in G$ to say $g \in V(G)$. We use $a \sim_G b$ to denote that vertices $a, b$ are adjacent in the graph $G$; when the context for $G$ is clear, we will simply write $a \sim b$.

We will call the specific distribution of pebbles distributed on the vertices of $G$ a {\it configuration}. 
The number of pebbles in a configuration $C$ is its {\it size}, denoted $|C|$.
At times, we will consider a variant called $k$-fold pebbling. 
In this variant, Player 2 wins only if he or she can move $k$ pebbles to the root. 
The {\it k-fold pebbling number} of $G$, denoted $\pi_k(G)$, is the smallest number of pebbles such that Player 2 can win the $k$-fold variant regardless of the initial distribution. 
For a graph $G$ with a given root $r$, we say a configuration is $k$-fold $r$-solvable (or just $r$-solvable when $k =1)$ if there is a sequence of pebbling moves such that $k$ pebbles can eventually reach $r$, and we will let $\pi(G, r)$ denote the minimum number of pebbles to guarantee that $G$ is $r$-solvable. 
Therefore $\pi(G) = \max_r \pi(G,r)$.

\begin{defi}[Graphs with specific pebbling properties] Let $G$ be a connected graph. 
We define the following based on the specific properties with regard to pebbling.

\item[] {\bf Class~0 graphs.}
If $\pi(G) = |G|$, then $G$ is \emph{Class~0}.

\item[] {\bf Frugal graphs.} 
If $G$ has diameter $\diam(G)$ and, for all $k\ge 1$, $\pi_k(G) \leq \pi(G)+(k-1)2^{\diam(G)}$ then $G$ is \emph{frugal}.
\end{defi} 

For most graphs, it is unknown whether or not they are frugal. 
However, complete graphs and graphs with diameter 2 are frugal \cite{herscovici2013t}, as are 
2-paths \cite{AGH2015path}, and semi-2-trees \cite{AGH2017trees}. 
Some of our results can be slightly improved in the case of frugal graphs.

Throughout we will make use of the following pebbling number facts.

\begin{fact}
\label{trivial} 
For any connected graph $G$, the following hold.
\begin{enumerate}[(i)]
\item 
\label{LowerVerts}
$\pi(G) \ge |G|$.
\item 
\label{LowerDiam}
$\pi(G) \ge 2^{{\rm diam}(G)}$ where ${\rm diam} (G)$ is the diameter of the graph $G$.
\item 
\label{UpperExpo}
$\pi(G) \le 2^{|G|-1}$.
\end{enumerate} 
Furthermore, these bounds are sharp.
\end{fact}

\begin{proof}
The lower bound (\ref{LowerVerts}) follows from constructing an $r$-unsolvable configuration by placing one pebble at each vertex except $r$.
The lower bound (\ref{LowerDiam}) follows from constructing an $r$-unsolvable configuration by placing $2^{{\rm diam}(G)}-1$ pebbles at a vertex of maximum distance from an appropriately chosen $r$.
The upper bound (\ref{UpperExpo}) follows from using Lemma~\ref{spanningSubgraph} (from Fact 3 of  \cite{chung1989pebbling}) along with the formula for the pebbling number of a tree (Fact 11 of \cite{chung1989pebbling}).
Indeed, for $T$ some breadth-first search spanning tree of $G$, we have $\pi(T)=\sum_{i=1}^t 2^{a_i} - t + 1$ for some $a_1\ge \cdots \ge a_t$ and some $t\ge 0$, where $\sum_{i=1}^t a_i=|G|-1$.
Then $\pi(G) \leq \pi(T) = \sum_{i=1}^t 2^{a_i} - t + 1\le 2^{\sum_{i=1}^t a_i} = 2^{|G|-1}$.

Complete graphs are sharp for (\ref{LowerVerts}), and paths are sharp for (\ref{LowerDiam}) and (\ref{UpperExpo}).
\end{proof}

Additionally, the following lemma is helpful for comparing the pebbling number of a graph against a subgraph. 
\begin{lemma}\label{spanningSubgraph}
{\normalfont\cite{chung1989pebbling}} If $G'$ is a spanning subgraph of a connected graph $G$ then $\pi(G')\geq \pi(G)$.
\end{lemma}

For two graphs, $G$ and  $H$, $g \in V(G)$, and $h \in V(H)$ we will use the simplistic notation $g \times H$ to denote the set $\{g\} \times V(H)$; likewise for $G \times h \colon V(G) \times \{h\}$.

\begin{defi}[Graph Products and Constructions]

For graphs $G$ and $H$, we define the following:


\item[] {\bf Box (or Cartesian) Product.}
$G \Osq H$ is the graph with $V(G \Osq H) = V(G) \times V(H)$ and $(g, h) \sim_{G \Osq H} (g', h')  \text{ whenever } (g=g' \text{ and } h\sim_H h' ) \text{ or } (h=h' \text{ and } g\sim_G g')$. 
For example, $K_2 \Osq K_2 = C_4$.


\item[] {\bf Strong Product.}
$G \boxtimes H$ is the graph with $V(G \boxtimes H) = V(G) \times V(H)$ and $(g, h) \sim_{G \boxtimes H} (g', h')  \text{ whenever one of the following holds }$
\begin{itemize}
\item $(g=g' \text{ and } h\sim_H h' )$, 
\item $(h=h' \text{ and } g\sim_G g'), \text{ or }$ 
\item  $(g\sim_G g' \text{ and } h\sim_H h')$.
\end{itemize}
For example, $K_2 \boxtimes K_2 = K_4$.


\item[] {\bf Cross (or Categorical) Product.}
$G \times H$ is the graph with $V(G \times H) = V(G) \times V(H)$ and  $(g, h) \sim_{G \times H} (g', h')   \text{ whenever } (g \sim_G g' \text{ and } h\sim_H h' )$. 
For example, $K_2 \times K_2 = 2K_2$.

%

\item[] {\bf Corona.}
$G\bowtie H$ is the graph constructed by taking one copy of $G$ and $|G|$ disjoint copies of $H$ and then, for each vertex in $G$, adding all possible edges to one distinct copy of $H$. 
Specifically, it has vertex set  $V(G\bowtie H) = V(G) \cup (V(G) \times V(H))$ and 
\begin{itemize}
\item $g \sim_{G\bowtie H} g'$ whenever $g \sim_G g'$,  
\item $g \sim_{G\bowtie H} (g, h')$ for all $h'\in H$, and 
\item $(g,h) \sim_{G\bowtie H} (g, h') $ whenever $h \sim_H h'$.
\end{itemize} For example, $K_2 \bowtie K_2 = 2K_3$ plus an edge (See Figure~{\normalfont\ref{CoronaEx}}). 

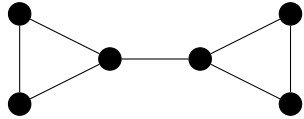
\begin{figure}[htb]
\begin{center}
\begin{tikzpicture}[scale=1.2]
\tikzstyle{every node}=[draw,circle,fill=black,minimum size=2pt,inner sep=3pt]
\path (0.0,0.5) node (a) {};
\path (0.0,-0.5) node (b) {};
\path (1.0,0.0) node (c) {};
\path (2.0,0.0) node (d) {};
\path (3.0,0.5) node (e) {};
\path (3.0,-0.5) node (f) {};
\draw (c)
  -- (b)
  -- (a)
  -- (c)
  -- (d)
  -- (e)
  -- (f)
  -- (d)
  ;
\end{tikzpicture}
\end{center}
\caption{$K_2 \bowtie K_2$} 
\label{CoronaEx}
\end{figure}


\end{defi}


\section{Pebbling on Box Products}\label{box}

In this section we prove the following.

\begin{theorem}\label{2graham}
Let $G$ and $H$ be connected graphs. 
Then
\[ \pi(G \Osq H) \leq 2 \pi(G) \pi(H). \]
\end{theorem}

The main idea behind the proof of this bound on the pebbling number for $G\Osq H$ can be seen in Figure~\ref{boxProduct}. 
Given a configuration of pebbles, we first move all of those pebbles over to the copy of $G$ that contains the target vertex (that is the top right vertex in Figure~\ref{boxProduct}). 
Now that all pebbles are in the copy of $G$ containing the target vertex, it remains to move all of the pebbles in this copy of $G$ to the target vertex. 
The goal is to have $\pi(G)$ pebbles in this copy of $G$ which will guarantee a pebbling solution. 

\begin{figure}[htb]
\begin{center}
\includegraphics[scale=1]{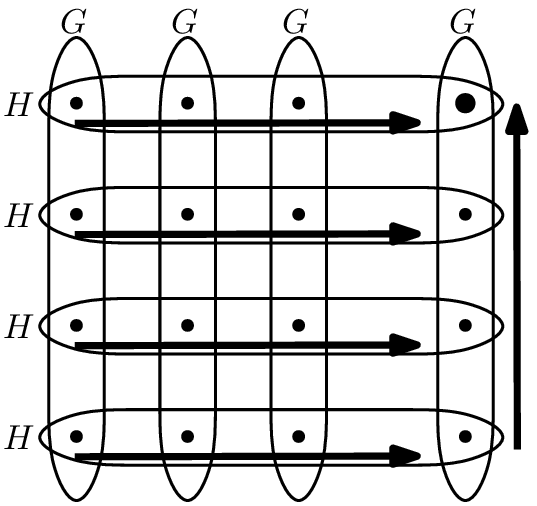}
\end{center}
\caption{$G\Osq H$ with pebble movements related to 
Proposition~\ref{boxProd}}\label{boxProduct}
\end{figure}

The following result is a proof of the idea above and is a stronger result than Theorem \ref{2graham} that follows. 

\begin{prop}\label{boxProd}
Let $G$ and $H$ both be connected graphs. 
Then,
\[ \pi(G \Osq H) \leq (\pi(G)+|G|) \pi(H).\] 
\end{prop}
\begin{proof}
Let $C$ be the initial configuration of $(\pi(G)+|G|) \pi(H) $ pebbles. 
For each vertex $g \in G$ let $p_{g}$ denote the total number of pebbles on the vertices in $g \times H$.
Notice that if $p_g \ge k \pi(H)$, then $k$ pebbles can be moved to any single vertex in $g \times H$. 
Hence, if $\sum_{i \in G} \lfloor \frac{p_i}{\pi(H)} \rfloor \ge \pi(G)$ then a total of $\pi(G)$ pebbles can be moved to $G \times h$ for any vertex $h \in H$. 
From there, pebbles can be moved within $G \times h$ to any desired vertex. 
Therefore, since $C$ is arbitrary, it suffices to show that $\sum_{i \in G} \left\lfloor \frac{p_i}{\pi(H)} \right\rfloor > \pi(G)$.

We have:
\begin{eqnarray*}
(\pi(G)+|G|) \pi(H) &=& \sum_{i \in G} p_i \\
  &=&  \pi(H) \sum_{i \in G} \frac{p_i}{\pi(H)}  \\
  &=&  \pi(H) \left(\sum_{i \in G} \left\lfloor \frac{p_i}{\pi(H)} \right\rfloor  + \sum_{i \in G} \left\{ \frac{p_i}{\pi(H)} \right\}  \right) \\
&<&  \pi(H) \left( \sum_{i \in G} \left\lfloor \frac{p_i}{\pi(H)} \right\rfloor  + |G| \right)  
\end{eqnarray*}
where $\{x\}$ denotes the fractional part of $x$, and the strict inequality follows from the fact that each $\bigl\{ \frac{p_i}{\pi(H)} \bigr\} < 1$.
Therefore
 \[\pi(G) = \frac{(\pi(G)+|G|) \pi(H) }{\pi (H)} -|G|  < \sum_{i \in G} \left\lfloor \frac{p_i}{\pi(H)} \right\rfloor\ ,\]
which completes the proof.
\end{proof}

\begin{proof}[Proof of Theorem \ref{2graham}]
The result follows from applying Fact \ref{trivial}(\ref{LowerVerts}) that $|G| \le \pi(G)$ to Proposition \ref{boxProd}.
\end{proof}

By the same method, we can improve the previous bound for $\pi(G \Osq H)$ in the case for which one of the graphs is frugal.

\begin{prop}\label{tightBoxProd}
Let $G$ and $H$ both be connected graphs. 
If $H$ is frugal, then
\[
\pi(G\Osq H)\leq 2^{\diam(H)}\pi(G)+ |G|\pi(H).
\]
\end{prop}

\begin{proof}
Suppose $H$ is frugal with diameter $d$. 
Then for all $k$, $\pi_k(H) \le \pi(H)+(k-1)2^d$. 
Let $C$ be the initial configuration with $|G|\pi(H)+2^d\pi(G)$ pebbles. 
As in Proposition~\ref{boxProd}, for each vertex $g\in G$ let $p_g$ denote the total number of pebbles on the vertices in $g\times H$. 
Notice that if $p_g\geq \pi_k(H) \le \pi(H)+(k-1)2^d$ then $k$ pebbles can be moved to any single vertex in $g\times H$. 
Hence, if $\sum_{i\in G} \left\lfloor \frac{p_i-\pi(H)}{2^d}+1\right\rfloor \geq \pi(G)$, then a total of $\pi(G)$ pebbles can be moved to $G\times h$ for any vertex $h\in H$, finally allowing Player 2 to move within $G\times h$ to any desired vertex. 
Therefore, since $C$ is arbitrary, it suffices to show that $\sum_{i\in G} \left\lfloor \frac{p_i-\pi(H)}{2^d}+1\right\rfloor\geq \pi(G)$.

We have:
\begin{eqnarray*}
|C|-|G|(\pi(H)-2^d) &=& \sum_{i\in G} (p_i-\pi(H)+2^d) \\
&=& 2^d \sum_{i\in G} \left(\frac{p_i-\pi(H)+2^d}{2^d}\right) \\
&=& 2^d \left(\sum_{i\in G} \left\lfloor \frac{p_i-\pi(H)+2^d}{2^d}\right\rfloor +\sum_{i\in G} \left\{ \frac{p_i-\pi(H)+2^d}{2^d}\right\}\right) \\
&<& 2^d\left( \sum_{i\in G} \left\lfloor \frac{p_i-\pi(H)+2^d}{2^d}\right\rfloor+|G|\right).
\end{eqnarray*}
Therefore
\[
\pi(G)= \frac{|C|-|G|(\pi(H)-2^d)}{2^d}-|G| < \sum_{i\in G} \left\lfloor \frac{p_i-\pi(H)+2^d}{2^d}\right\rfloor,
\]
and the result follows.
\end{proof}

Since $2^d\leq \pi(H)$, we have $2^d\pi(G)+|G|\pi(H)\leq \pi(H)\pi(G)+|G|\pi(H)$.
Thus Proposition~\ref{tightBoxProd} is at least as strong as the bound in Proposition~\ref{boxProd} as long as  $\min(\{(\pi(G)+|G|)\pi(H), (\pi(H)+|H|)\pi(G)\})=(\pi(G)+|G|)\pi(H)$.

One consequence of Graham's conjecture concerns the pebbling number of the product $G^{\Osq k} = \underbrace{G \Osq \ldots \Osq G}_{ k \text{ factors } }$. 
If the conjecture holds, then $\pi (G^{\Osq k}) \le \pi(G)^k$. 
As shown in \cite{chung1989pebbling} this would be sharp, as the pebbling number of the hypercube $Q_k = K_2^{\Osq k}$ is $2^k$.  
In fact, using Proposition~\ref{boxProd}, we show the following.

\begin{cor} \label{almost}
For any connected graph $G$ and any positive integer $k$,
\[ \pi(G^{\Osq k})^{1/k} < \pi(G) + |G|. \]
\end{cor}

\begin{proof}
The result is true for $k=1$, and for $k>1$ we use induction and Proposition \ref{boxProd}, which yields
$\pi (G^{\Osq k}) 
< \pi (G^{\Osq (k-1)})  (\pi(G)+|G|)
\leq (\pi(G)+|G|)^k.$
\end{proof}



Since, for many graphs $\pi(G) \gg |G|$ (for example, paths and other high diameter graphs with Fact \ref{trivial}(\ref{LowerDiam})), the exponential bound for the growth of $\pi(G^{\Osq k})$ in Corollary \ref{almost} is close to best possible.



\section{Pebbling on Strong Products} \label{strong}

In this section, we use similar techniques to prove the following.

\begin{theorem}\label{32stronggraham}
Let $G$ and $H$ be connected graphs. 
Then,
\[ \pi(G\boxtimes H) \le \frac{3}{2} (\pi(G)+1)(\pi(H)+1). \]
\end{theorem}


The key difference in our approach between the box product and the strong product is the fact that, occasionally, a pebble can move in {\it both} $G$ and $H$. 
Specifically, in the proof of Proposition \ref{boxProd}, all of the pebbles were first moved through copies of $H$, then through a copy of $G$. 
However, in the strong product case, we will be able to ``cut the corner'' for one of the moves as illustrated in Figure~\ref{strongProduct}.

\begin{figure}[htb]
\begin{center}
\includegraphics[scale=1]{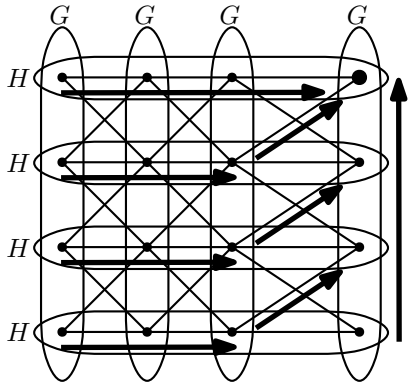}
\end{center}
\caption{$G\boxtimes H$ with pebble movements related to Proposition~\ref{strongProd}}\label{strongProduct}
\end{figure}

To gain the most of this ``cut the corner''  move, we introduce the following function. 
Let $\phi(G)$ be the pebbling number of a connected graph $G$ with the following additional rule: before any normal pebbling moves are made, each pebble may make a single pebbling move without any cost of losing a pebble. 
These moves must be done simultaneously, and a pebble is still considered to reach the root as a result of this move. 
We first show that this ``cut the corner'' move, indeed, saves nearly a factor of $2$ of the pebbles.

\begin{lemma}\label{freeMove} 
Let $\phi(G)$ be defined as directly above.
Then, for any connected graph $G$, we have
\[ \phi(G)\leq \left\lceil\frac{\pi(G)}{2}\right\rceil . \]
\end{lemma}

\begin{proof}
Let $r$ be the root. 
Let $C$ be a configuration of size $\left\lceil\frac{\pi(G,r)}{2}\right\rceil$ on $G$. Copy each pebble in $C$ to create another configuration $D$ of size at least $\pi(G,r)$; then $D$ is $r$-solvable under standard pebbling.
Now let $D'$ be a non-trivial subconfiguration of $D$ that is minimally $r$-solvable. 
That is, it solves $r$ with the fewest number of steps and with no remaining pebbles. 
Let $\sigma'$ be its corresponding $r$-solution.

We note that $D'$ has an even number of pebbles at each vertex.
Indeed, if some vertex $v$ had an odd number of pebbles then one of its pebbles $p$ would need to wait for another pebble to arrive to $v$ in order to be involved in a pebbling step, which takes more steps than if $p$ already had its partner at $v$, thus contradicting the minimality of $D'$.

Every pebble in $D'$ has a first move in $\sigma'$ (since otherwise $D'$ would not include that pebble), so let $S'$ be the set of all such first moves. 
Because the pebbles in $D'$ come in pairs at each vertex, we can assume that each pebble in $D'$ is paired with another pebble on its initial vertex for its first move, and we can start $\sigma'$ by making all the $S'$ moves. 
Let $D^*$ be the resulting configuration, which of course is $r$-solvable. 

Now back up to think about $D'$ differently.
Merge the pairs of pebbles in $D'$ to create the sub-configuration $C'$ of $C$. 
Make the first moves $S$ with no cost that correspond to the costly moves $S'$, and let $C^*$ be the resulting configuration. 
Notice that $C^*=D^*$, and hence is $r$-solvable.
\end{proof}

\begin{prop}\label{strongProd}
Let $G$ and $H$ both be connected graphs. 
Then
\[\pi(G\boxtimes H) \le \frac{1}{2} (\pi(G)+2|G|+1)(\pi(H)+1).\]
\end{prop}

\begin{proof}
We will follow a similar proof to that posed in Proposition~\ref{boxProd}. 
Let $(r_G,r_H)$ be the root,  $C$ be a configuration of at least $\frac{1}{2} (\pi(G)+2|G|+1)(\pi(H)+1)$ pebbles, and for each vertex $g \in G$ let $p_{g}$ denote the total number of pebbles on the vertices in $g \times H$.

As with the proof of Proposition~\ref{boxProd}, we want to show that sufficiently many pebbles can reach some vertex in $G\times r_H$. 
The key to the proof is that the number of pebbles needed to move onto the subgraph $G \times r_H$ is sufficiently less than before.
Let $N(g)$ denote the {\it open neighborhood} $\{g'\in G\mid g'\sim_G g\}$,
and define the {\it closed neighborhood} $N[g]=N(g)\cup\{g\}$.

{\it Claim.} 
If $p_g \geq k (\pi(H)+1)$, then a total of $k$ pebbles can be moved to any configuration on $N[g] \times H$. 

{\it Proof of Claim}. 
Partition the $p_g$ pebbles into $k$ equitable parts (or as equitable as possible). We show that each part will allow for at least one pebble to reach any vertex $(g', h)  \in N[g] \times H$. 
For each part, observe that using only those pebbles either (a) one pebble can be moved onto $(g,h)$ using only those pebbles or (b) there are two pebbles at $(g, h)$. 
In case (a), notice that, for the final step of the corresponding $h$-solution, the pebble may be moved to any $(g', h)$ (instead of $(g, h)$)  
In case (b), the two pebbles at $(g, h)$ can be used to make a pebbling move to $(g', h)$ for any $g \sim_G g'$. 
$\hfill\diamondsuit$


As a result of the claim together with Lemma \ref{freeMove}, it follows that,
whenever $\sum_{g \in V(G)} \left\lfloor \frac{p_g}{\pi(H) + 1} \right\rfloor$ 
$> \lceil \pi(G) / 2 \rceil$, there is an $(r_G, r_H) $-solution.



We have:
\begin{eqnarray*}
\frac{(\pi(G)+2|G|+1)(\pi(H)+1)}{2}&\leq& \sum_{i \in G} p_i \\
  &=&  (\pi(H)+1) \sum_{i \in G} \frac{p_i}{\pi(H)+1}  \\
  &=&  (\pi(H)+1) \left(\sum_{i \in G} \left\lfloor \frac{p_i}{\pi(H)+1} \right\rfloor  + \sum_{i \in G} \left\{ \frac{p_i}{\pi(H)+1} \right\}  \right) \\
&<&  (\pi(H)+1) \left( \sum_{i \in G} \left\lfloor \frac{p_i}{\pi(H)+1} \right\rfloor  + |G| \right)\ .  
\end{eqnarray*}
Therefore,
\[
  \left\lceil\frac{\pi(G)}{2}\right\rceil \le \frac{(\pi(G)+2|G|+1)(\pi(H)+1)}{2(\pi(H)+1)}-|G| < \sum_{i\in G} \left\lfloor \frac{p_i}{\pi(H)+1}\right\rfloor\ ,
\]
where the first inequality follows from the properties of the ceiling function, and the second inequality follows directly from above. 
This completes the proof.
\end{proof}

\begin{proof}[Proof of Theorem \ref{32stronggraham}]
The proof follows from Proposition \ref{strongProd} and Fact \ref{trivial}(\ref{LowerVerts}).
\end{proof}

Upon comparing the proof of Proposition \ref{strongProd} to the proof of Proposition \ref{boxProd}, a ``shortcut'' is taken, saving one move, and up to a factor of 2 of the pebbles. 
As a result, it may seem reasonable to conjecture that $\pi(G\boxtimes H) \le \frac{1}{2}\pi(G) \pi(H)$. 
However, $\pi(K_2 \boxtimes P_3) = 6$ whereas $\pi(K_2) = 2$ and $\pi(P_3) = 4$. 
The gains made by the shortcut are lost by the fact that there is now more room to place unused pebbles. 
As a result, we conjecture the following variation to Graham's conjecture for strong products.

\begin{conj}
For any connected graphs $G$ and $H$,
\[ \pi(G\boxtimes H) \le   \max \left\{ \frac{1}{2}\pi(G)\pi(H), |G||H| \right\}+2. \]
\end{conj}

We note that $\pi(K_2 \boxtimes P_n) = 2^{n-1}+2$ if $n\geq 2$ whereas $\pi(K_2) =2, \pi(P_4)=8$.

\section{Pebbling on Cross Products}\label{cross}
%

Previously, our main technique has been to push pebbles in a canonical manner. 
However, in the cross product $G \times H$, $G$ (or $H$) is not necessarily a subgraph. 
Hence, using the structure of $G$ or $H$ for pebbling within $G \times H$ appears to be daunting. 
As exemplified by $K_2 \times K_2$, $G \times H$ is not connected when both $G$ and $H$ are bipartite and neither $G$ nor $H$ is $K_1$.

Our main method to overcome these obstacles is to consider connected spanning bipartite subgraphs of $G$ and $H$. This can be seen in Figure~\ref{crossProdPic}.

Such subgraphs exist provided the graphs are connected; for instance, consider any spanning tree. 
We will prove the following:

\begin{theorem}\label{bicross}
Let $G ~(\ne K_1)$ and $H$ be connected graphs with $H$ nonbipartite. 
For $G$ and $H$, choose any connected spanning bipartite subgraphs, $G'$ and $H'$, respectively. 
Then,
\[ \pi(G\times H) \le 2( \pi(G')+|G|) \pi(H')^2. \]
\end{theorem}

\begin{figure}[t]
\begin{center}
\includegraphics[scale=1]{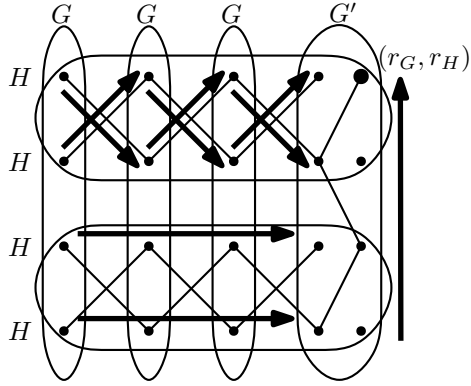}
\end{center}
\caption{$G\times H$ with pebble movements related to Theorem~\ref{bicross}}\label{crossProdPic}
\end{figure}

To prove this theorem, we will consider the special case $K_2 \times H$ where $H$ is nonbipartite.

\begin{lemma}\label{k2cross}
Let $H$ be a connected nonbipartite graph and let $H'$ be a connected spanning bipartite subgraph. 
Then
\[ \pi(K_2\times H) \le 2\pi(H')^2. \]
\end{lemma}

\begin{proof}
Since $H$ is nonbipartite, we can choose an edge, $\{a,b\}$, that is  in $H'$ but not $H$, and $a$ and $b$ are in the same partite set. 
Since $H'$ is a connected spanning bipartite subgraph of $H$, we can construct a connected bipartite spanning subgraph of $K_2 \times H$, using two copies ($H'_0$ and $H'_1$) of $H'$ and adding a single edge $\{a', b'\}$ with $a'$ being the copy of $a$ in $V(H'_0)$ and $b'$ being the copy of $b$ in $V(H'_1)$. 

Choose a root $r \in V(K_2 \times H)$ and, without loss of generality, let $r \in V(H'_0)$. 
For every 2 pebbles that reach $b'$, we can move one pebble to $a'$. 
It follows that we can move one pebble to $a'$ for every $2\pi(H')$ pebbles in $H'_1$. 

Let $y$ and $z$ denote the number of pebbles initially in $H'_0$ and $H'_1$ respectively. 
By pebbling in $H' _1$ through $\{a',b'\}$ and then through $H'_0$ in this manner, we see that, if  $y + \frac{z}{2\pi(H')}\ge \pi(H')$, then one pebble can necessarily reach $r$. 
Therefore, since $y,z \ge 0$ and $y+z = 2\pi(H')^2$, we have $y + \frac{z}{2\pi(H')}\ge \frac{y+z}{2\pi(H')}= \pi(H')$, which finishes the proof. 
\end{proof}

We are now ready to prove the main result of this section.

\begin{proof}[Proof of Theorem \ref{bicross}]

Let $(r_G,r_H)$ be the root and $C$ be a configuration of $2(\pi(G')+ |G|)\pi(H')^2$ pebbles. 
For each vertex $g \in G$ choose an edge $e_g \in E(G)$ such that $g \in e_g$. 
Arbitrarily order the vertices of $G$, and let $p_g$ denote the number of pebbles on $V(e_g \times H)\setminus \bigl(\bigcup_{g'>g}V(e_{g'} \times H)\bigr)$, 
so that each pebble is counted once, and $\sum_g p_g$ is exactly the total number of pebbles.

Note that, since $G'$ is a connected spanning bipartite subgraph of $G$, we can find $G'$ as a subgraph of $G \times H$ containing $(r_G,r_H)$ among the vertices of $G \times \{r_H, x\}$ for any edge $\{r_H, x\}$ in $H$ (such an edge exists because $H$ is connected). 
Further, for every $g \in G$, the vertices of this subgraph $G'$ will contain at least one vertex in $e_g\times H$. 

By Lemma \ref{k2cross}, every $2\pi(H')^2$ pebbles within $e_g \times H$ can be used to move a single pebble onto any vertex in $e_g \times H$.
In particular, if $p_g \ge 2 k \pi(H')^2$, then $k$ pebbles can be moved to any single vertex in $e_g \times H$. 
Since $e_g \times H$ intersects our chosen copy of $G'$ in at least one vertex, if $\sum_{i \in G} \lfloor \frac{p_i}{2\pi(H')^2} \rfloor \ge \pi(G')$, a total of $\pi(G')$ pebbles can be moved to our chosen subgraph $G'$. 
From there, pebbles can be moved within $G'$ to the root. 
Therefore, since $C$ is arbitrary, it suffices to show that $\sum_{i \in G} \left\lfloor \frac{p_i}{2\pi(H')^2} \right\rfloor \ge \pi(G')$.

We have:
\begin{eqnarray*}
2(\pi(G')+ |G|)\pi(H')^2  &=& \sum_{i \in G} p_i \\
  &=&  2\pi(H')^2 \sum_{i \in G} \frac{p_i}{2\pi(H')^2}  \\
  &=&  2\pi(H')^2 \left(\sum_{i \in G} \left\lfloor \frac{p_i}{2\pi(H')^2} \right\rfloor  + \sum_{i \in G} \left\{ \frac{p_i}{2\pi(H')^2} \right\}  \right) \\
&<&  2\pi(H')^2 \left( \sum_{i \in G} \left\lfloor \frac{p_i}{2\pi(H')^2} \right\rfloor  + |G| \right).  
\end{eqnarray*}
Therefore
\[
  \pi(G') = \frac{2(\pi(G')+ |G|)\pi(H')^2}{2 \pi(H')^2} -|G| < \sum_{i \in G} \left\lfloor \frac{p_i}{2\pi(H')^2} \right\rfloor  + |G| -|G| = \sum_{i \in G} \left\lfloor \frac{p_i}{2\pi(H')^2} \right\rfloor. 
\]
This completes the proof.
\end{proof}

\begin{cor}
Let $G$ $(\ne K_1)$ and $H$ be connected graphs with $H$ nonbipartite. 
For $G$ and $H$, choose any connected spanning bipartite subgraphs, $G'$ and $H'$, respectively. 
Then,
\[ \pi(G\times H) \le 4 \pi(G') \pi(H')^2. \] \hfill $\qed$
\end{cor}

It is worth noting that, for these results, the choice of which graph is $G$ and which graph is $H$ is important. 
In fact, the condition that $H$ is nonbipartite is applied subtly in Lemma \ref{k2cross}. 
Hence, with these methods, a nonbipartite graph must take the role of $H$.

The concept of using spanning bipartite subgraphs is undoubtedly unaesthetic. 
Indeed, we believe that a similar type of inequality should hold without the use of bipartite subgraphs.

\begin{conj}
Let $G$ and $H$ be connected graphs with $G \ne K_1$ and $H$ nonbipartite. 
Then,
\[ \pi(G\times H) \le \frac{9}{16} \pi(G) \pi(H)^2. \]
\end{conj}

If this conjecture is true, it would be asymptotically tight. 
It is known that $\pi(C_{2k}) = 2^k $ and $\pi(C_{2k+1}) = 2 \lfloor 2^{k+1}/3 \rfloor+1$ \cite{cycle}. 
Note that $K_2 \times C_{2k+1} = C_{4k+2}$, so $\pi(K_2) \pi(C_{2k+1})^2 = 2 \left( 2\lfloor 2^{k+1}/3 \rfloor + 1 \right)^2 \sim \frac{16}{9}~2^{2k+1} = \frac{16}{9} \pi(C_{4k+2})$ as $k\to \infty$. 
Further, this example is persuasive for showing that the $\pi(H)^2$ component, as opposed to just $\pi(H)$, is likely necessary in a Graham-type bound for the cross product.

\section{Pebbling on Coronas}\label{corona}

For an idea of how the pebbles move in these graphs, see Figure~\ref{coronaSketch}. 
The following two results are straightforward from the structure of $G\bowtie H$.

\begin{figure}[htb]
\begin{center}
\includegraphics[scale=1]{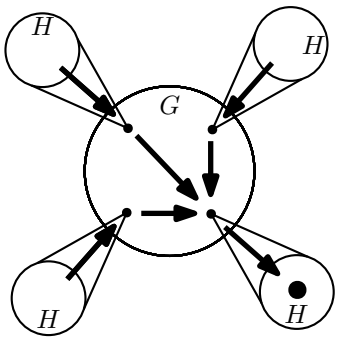}
\end{center}
\caption{$G\bowtie H$ with pebble movements}\label{coronaSketch}
\end{figure}

\begin{lemma}\label{2pig}
Let $G$ be a connected graph and $H$ be any graph.
If there are $2\pi(G)$ pebbles on the vertices in $G$ then there is an $r$-solution in $G\bowtie H$ for any $r$.
Furthermore, if $r$ is in a copy of $H$ and $\pi(G)+1$ pebbles are on the vertices in $G$ with at least one pebble on the vertex in $G$ adjacent to $r$, then a pebble can be moved to $r$.
\end{lemma}

For a vertex $v\in V(G\bowtie H)-V(G)$ denote by $H^v$ the copy of $H$ in $G\bowtie H$ that contains $v$.

\begin{lemma}\label{42pebbles}
Let $G$ be a connected graph and $H$ be any graph.
Suppose that $r\in V(G\bowtie H)-V(G)$. 
If there are $4$ pebbles at a vertex in $H^r$ or there are $2$ pebbles at two distinct vertices in $H^r$, then there is an $r$-solution in $G\bowtie H$.
\end{lemma}

Lemma~\ref{42pebbles} implies an $r$-unsolvable configuration has at most $|H|+1$ pebbles in $H^r$. 

\begin{lemma}\label{2k}
Let $G$ be a connected graph, $H$ be any graph, and consider the graph $G\bowtie H$. 
If there are $|H|+2k-1$ pebbles on the vertices in one copy of $H$ in $G\bowtie H$, say $H^*$, then $k$ pebbles can be moved to the vertex in $G$ that is adjacent to all vertices in $H^*$. 
\end{lemma}

\begin{proof}
Given a configuration $C$ of $|H|+2k-1$ pebbles on $H^*$, at most $|H|$ (``odd'') vertices have an odd number of pebbles on them.
By removing a pebble from each odd vertex we obtain a configuration $C'$ of at least $2k$ pebbles, for which each vertex is even.
By pairing up pebbles on each vertex, at least $k$ pebbling moves can be made.
\end{proof}


\begin{theorem}
\label{coronaProdH}
Let $G$ be a connected graph and $H$ be any graph.
Then $\pi(G\bowtie H)\leq |G||H| + 4\pi(G)$.
\end{theorem}

\begin{proof}
Let $r$ be the root and suppose that $C$ is an $r$-unsolvable configuration of maximum size.
Write $g=|G|$, $h=|H|$, and $H_1,\ldots,H_g$ for the copies of $H$ in $G\bowtie H$, corresponding to the vertices $v_1,\ldots,v_g$ in $V(G)$.
For a subgraph $F\subseteq G\bowtie H$ let $C_F$ be the restriction of $C$ to $F$.

Suppose that $r\in V(G)$.
In this case we know that $|C_G|=0$ because, otherwise, any pebble on a vertex $v_i\in V(G)$ can be moved back to two pebbles in $H_i$, preserving $r$-unsolvability but contradicting maximality.
Similarly, each $|C_{H_i}|\ge h$ since, otherwise, we could add a pebble to some vertex in $H_i$ that doesn't yet have one.
Let $k_i\ge 0$ be such that $|C_{H_i}|=h+2k_i-1+\epsilon_i$, where $\epsilon_i\in\{0,1\}$.
Then, by Lemma \ref{2k}, the total number of pebbles that all the copies of $H$ contribute to $G$ is at least $\sum_{i=1}^g k_i < \pi(G)$.
Hence $|C| = \sum_{i=1}^g |C_{H_i}| \le \sum_{i=1}^g (h+2k_i) \le gh + 2(\pi(G)-1)$.

Now suppose, without loss of generality, that $r\in V(H_1)$.
As above, $|C_G|\le 1$ (there could be a pebble on $v_1$).
Also, $|C_{H_i}|\ge h$ for $i>1$, while Lemma \ref{42pebbles} implies that $|C_{H_1}|\le h+1$.
Defining $k_i$ as above for $i>1$, Lemmas~\ref{2pig} and \ref{2k} show that $\sum_{i=2}^g k_i < \pi(G)$ when $|C_G|=1$ and that $\sum_{i=2}^g k_i < 2\pi(G)$ when $|C_G|=0$.
Therefore, if $|C_G|=0$ then $|C| = |C_{H_1}| + \sum_{i=2}^g |C_{H_i}| \le (h+1) + \sum_{i=2}^g (h+2k_i) \le gh+1 + 2(2\pi(G)-1) = gh + 4\pi(G) - 1$, and if 
$|C_G|=1$ then $|C|=|C_G|+|C_{H_1}|+\sum_{i=2}^g |C_{H_i}|\leq 1+(h+1)+\sum_{i=2}^g(h+2k_i)\leq gh+2+2(\pi(G)-1=gh+2\pi(G)$.

Therefore, if $|C| \ge gh + 4\pi(G)$, $C$ is $r$-solvable.
\end{proof}

A well-known example of a corona is the {\it sun} $S_n=K_m\bowtie K_1$, where $n=2m$.
Since $S_n$ is a split graph, an application of the main result of \cite{AGH2014split} yields $\pi(S_n)=3m+2$ for $m \ge 2$, while Theorem~\ref{coronaProdH} only gives the upper bound $5m$.
This shows that the bound can be weak.
However, the bound can also be fairly sharp.
Indeed, one can see that  for $m \ge 2$, $\pi(K_m\bowtie K_t)=mt+2m+2$, whereas this theorem gives the upper bound $mt+4m$, which is asymptotically sharp in $t$.
Also, for any non-complete connected graph $H$, $K_1\bowtie H$ has diameter two and so $\pi(K_1\bowtie H)=|H|+1$ (it is Class 0 by \cite{CHH1997diam2}), while this theorem gives the upper bound $|H|+4$.


\section{Notes}
In Section \ref{strong}, we introduced the new function $\phi$ that proved useful in deriving Proposition~\ref{strongProd}.
It would therefore be of use to study $\phi$ for various graphs and graph classes in order to sharpen this bound and investigate its tightness.
For instance, the bound is sharp for $C_5$, $\pi(C_5)=5$ and $\phi(C_5)=3$, but denser graphs have much looser inequalities, i.e., $\phi(K_n)=1$.  The case when the diameter is 2 also presents an interesting question; for while it may be tempting to say $\phi(G)=2$ when $\diam(G)=2$, a quick check has that the Petersen Graph, $P$ has $\phi(P) = 3$.

On the subject of frugal graphs, and based on the results of \cite{AGH2015path,AGH2017trees}, it would be worth exploring whether or not all chordal graphs are frugal.

Finally, the truth of Graham's conjecture would imply that the set of Class~0 graphs is closed under cartesian products.
This should be a robust direction to pursue, especially if we also add the frugal property, although it has been suggested in \cite{hurlWeight} that the square of the Lemke graph might be a counterexample.



\begin{thebibliography}{1}

\bibitem{AGH2014split}
L. Alc\'on, M. Gutierrez, and G. Hurlbert, Pebbling in split graphs, \textit{SIAM J. Discrete Math.} \textbf{28} (2014), no.~3, 1449--1466.

\bibitem{AGH2015path}
L. Alc\'on, M. Gutierrez, and G. Hurlbert, Pebbling in 2-paths, \textit{Elec. Notes Discrete Math.} \textbf{50} (2015), 145--150.

\bibitem{AGH2017trees}
L. Alc\'on, M. Gutierrez, and G. Hurlbert, Pebbling in semi-2-trees, \textit{Discrete Math.} \textbf{340} (2017), 1467--1480.

\bibitem{CHH1997diam2}
T. Clarke, R. Hochberg, and G. Hurlbert, Pebbling in diameter two graphs and products of paths, \textit{J. Graph Th.} \textbf{25} (1997), no.~2, 119--128.

\bibitem{chung1989pebbling}
F. Chung, Pebbling in hypercubes, \textit{SIAM J. Discrete Math.} \textbf{2} (1989), no.~4, 467--472.

\bibitem{czyghurl2006girth}
A. Czygrinow and G. Hurlbert, Girth, pebbling, and grid thresholds, \textit{SIAM J. Discrete Math.} \textbf{20} (2006), no.~1, 1--10.

\bibitem{feng2002pebbling}
R. Feng and J.Y. Kim, Pebbling numbers of some graphs, \textit{Sci. China Ser. A} \textbf{45} (2002), no.~4, 470--478.

\bibitem{herscovici2003graham}
D. Herscovici, Graham's pebbling conjecture on products of cycles, \textit{J. Graph Theory} \textbf{42} (2003), no.~2, 141--154.

\bibitem{herscovici2013t}
D. Herscovici, B. Hester, and G. Hurlbert, t-pebbling and extensions, \textit{Graphs Combin.} \textbf{29} (2013), no.~4, 955--975.


\bibitem{hurlHand}
G. Hurlbert, Graph Pebbling.
Handbook of Graph Theory, (2nd ed.), \textit{Discrete Mathematics and its Applications}, J. Gross, J. Yellen, and P. Zhang, eds., 1428--1454, CRC Press, Boca Raton, 2014.

\bibitem{hurlWeight}
G.  Hurlbert, The weight function lemma for graph pebbling, \textit{J. Combin. Opt.} \textbf{34} (2017), 343--361.


\bibitem{KimKim2}
J.Y. Kim and S.S. Kim, Pebbling Numbers of the Compositions of Two Graphs, \textit{J. Korea Soc. Math. Educ. Ser. B: Pure Appl. Math.} \textbf{9} (2002), no. 1, 57--61.

\bibitem{moews}
D. Moews, Pebbling graphs, \textit{J. Combin. Theory Ser. B} \textbf{55} (1992), no. 2, 244--252.

\bibitem{cycle}
L. Pachter, H. Snevily, and B. Voxman, On pebbling graphs, \textit{Congr. Numer.} \textbf{107} (1995), 65--80.

\bibitem{snevily20002}
H. Snevily and J. Foster, The 2-pebbling property and a conjecture of Graham's, \textit{Graphs Combin.} \textbf{16} (2000), no.~2, 231--244.

\bibitem{wang2001pebbling}
S.S. Wang, Pebbling and {G}raham's {C}onjecture, \textit{Discrete Math.} \textbf{226} (2001), nos.~1-3, 431--438.

\bibitem{wang2009graham}
Z. Wang, Y. Zou, H. Liu, and Z. Wang, Graham's pebbling conjecture on product of thorn graphs of complete graphs, \textit{Discrete Math.} \textbf{309} (2009), no.~10, 3431--3435.

\end{thebibliography}

\providecommand{\bysame}{\leavevmode\hbox to3em{\hrulefill}\thinspace}
\providecommand{\MR}{\relax\ifhmode\unskip\space\fi MR }
\providecommand{\MRhref}[2]{%
  \href{http://www.ams.org/mathscinet-getitem?mr=#1}{#2}
}
\providecommand{\href}[2]{#2}

\end{document}